\documentclass{amsart}
\usepackage{amsmath}
\usepackage{amsthm}
\usepackage{amssymb}
\usepackage{verbatim}
\theoremstyle{plain}
  \newtheorem{theorem}{\bf Theorem}[section]
  \newtheorem{proposition}[theorem]{\bf Proposition}
  \newtheorem{lemma}[theorem]{\bf Lemma}

\DeclareMathOperator{\codim}{codim}
\DeclareMathOperator{\loc}{loc}

\newcommand{\linearmapj}{L_j}
\newcommand{\linearfamily}{\mathbf{L}}

\newcommand{\nonlinearmapj}{B_j}

\newcommand{\id}{Id}

\begin{document}
%%%%%%%%%%%%%%%%%%%%%%%%%%%%%%%%%%%%%%%%%%%%%%%%%%%%%%%%%%%%%%%%%%%%%%%%%%%%%%%%%%%%%%%%%%
\title{Stability of the Brascamp--Lieb constant and applications} 
\subjclass[2000]{}
\author{Jonathan Bennett}
\author{Neal Bez}
\author{Taryn C. Flock}
\author{Sanghyuk Lee}
\thanks{This work was supported by the European Research Council [grant
number 307617] (Bennett, Flock), JSPS Kakenhi [grant numbers 26887008 and 16K13771] (Bez), NRF of Korea [grant number 2015R1A2A2A05000956] (Lee). \\This article appeared in the \textit{American Journal of Mathematics}, Volume 140, Issue 22, 2018, pages 543--569, \copyright 2018, Johns Hopkins University Press.}

\address{Jonathan Bennett and Taryn C. Flock: School of Mathematics, The Watson Building, University of Birmingham, Edgbaston,
Birmingham, B15 2TT, England.}
\email{J.Bennett@bham.ac.uk, T.C.Flock@bham.ac.uk}
\address{Neal Bez: Department of Mathematics, Graduate School of Science and Engineering,
Saitama University, Saitama 338-8570, Japan}
\email{nealbez@mail.saitama-u.ac.jp}
\address{Sanghyuk Lee: Department of Mathematical Sciences, Seoul National University, Seoul 151-747, Korea}
\email{shklee@snu.ac.kr}
%\date{27 June 2016}

%%%%%%%%%%%%%%%%%%%%%%%%%%%%%%%%%%%%%%%%%%%%%%%%%%%%%%%%%%%%%%%%%%%%%%%%%%%%%%%%%%%%%%%%%%%
\begin{abstract}
We prove that the best constant in the general Brascamp--Lieb inequality is a locally bounded function of the underlying linear transformations. As applications we deduce certain very general Fourier restriction, Kakeya-type, and nonlinear variants of the Brascamp--Lieb inequality which have arisen recently in harmonic analysis.
\end{abstract}

\maketitle
%%%%%%%%%%%%%%%%%%%%%%%%%%%%%%%%%%%%%%%%%%%%%%%%%%%%%%%%%%%%%%%%%%%%%%%%%%%%%%%%%%%%%%%%%%%%
\section{Introduction}
The celebrated Brascamp--Lieb inequality, which simultaneously generalises many important multilinear inequalities in analysis, including the H\"older, Loomis--Whitney and Young convolution inequalities, takes the form
\begin{eqnarray}\label{BL}
\int_{H} \prod_{j=1}^m (f_j \circ \linearmapj )^{p_j}
\leq C \prod_{j=1}^m \bigg(\int_{H_j} f_j\bigg)^{p_j}.
\end{eqnarray}
Here $m$ denotes a positive integer, $H$ and $H_j$ denote euclidean spaces of finite dimensions $n$ and $n_j\leq n$ respectively, equipped with Lebesgue measure for each $1\leq j\leq m$. The maps $\linearmapj:H\to H_j$  are surjective linear transformations, and the exponents $0\leq p_j\leq 1$ are real numbers.
This inequality is often referred to as multilinear, since it is equivalent to
\begin{eqnarray}\label{BLeq}
\int_{H} \prod_{j=1}^m f_j \circ \linearmapj
\leq C \prod_{j=1}^m \|f_j\|_{L^{q_j}(H_j)}
\end{eqnarray}
where $q_j=p_j^{-1}$ for each $j$.

Following the notation introduced in \cite{BCCT1} we denote by $\mbox{BL}(\linearfamily,\mathbf{p})$ the smallest constant $C$ for which \eqref{BL} holds for all nonnegative input functions $f_j\in L^1(\mathbb{R}^{n_j})$, $1\leq j\leq m$. Here $\linearfamily$ and $\mathbf{p}$ denote the $m$-tuples $(\linearmapj)_{j=1}^m$ and $(p_j)_{j=1}^m$ respectively. We refer to  $(\linearfamily,\mathbf{p})$ as the \textit{Brascamp--Lieb datum}, and $\mbox{BL}(\linearfamily,\mathbf{p})$ as the \textit{Brascamp--Lieb constant}.
To avoid completely degenerate cases, where $\mbox{BL}(\linearfamily,\mathbf{p})$ is easily seen to be infinite, it is natural to restrict attention to data $(\linearfamily,\mathbf{p})$ for which
$$
\bigcap_{j=1}^m \ker{\linearmapj}=\{0\}.
$$
In \cite{L} Lieb proved that $\mbox{BL}(\linearfamily,\mathbf{p})$ is exhausted by centred gaussian inputs
$$
f_j(x)=\exp(-\pi\langle A_jx,x\rangle),
$$
for arbitrary positive-definite transformations $A_j:H_j\rightarrow H_j$, and thus
$$
\mbox{BL}(\linearfamily,\mathbf{p})=\sup \;\frac{\prod_{j=1}^m(\det A_j)^{p_j/2}}{\det\left(\sum_{j=1}^m p_j\linearmapj^*A_j\linearmapj\right)^{1/2}},
$$
where
%$L^*$ is the transpose of L and
the supremum is taken over all such $A_j$, $1\leq j\leq m$. While this considerably reduces the complexity of computing the Brascamp--Lieb constant for a given datum, it does not provide a transparent characterisation of the data for which it is finite. This problem was addressed in \cite{BCCT1} and \cite{BCCT2} (see also the forerunner \cite{CLL} in the rank one setting), where it was shown that $\mbox{BL}(\linearfamily,\mathbf{p})$ is finite if and only if the scaling condition
\begin{equation}\label{scaling}\sum_{j=1}^mp_jn_j=n
\end{equation}
and the dimension condition
\begin{equation}\label{char}
\dim(V)\leq\sum_{j=1}^mp_j\dim(\linearmapj V)
\end{equation}
hold for all subspaces $V\subseteq H$.

In this note we turn our attention to the \emph{stability} of the constant $ \mbox{BL}(\linearfamily,\mathbf{p})$ as a function of the linear maps $\linearfamily$, establishing the following basic result:
%and show that for an initial input $(\linearfamily^0,\mathbf{p}^0)$ such that $ \mbox{\emph{BL}}(\linearfamily,\mathbf{p})<\infty$ not only is there an open neighbourhood of $(\linearfamily^0,\mathbf{p}^0)$ on which $\mbox{\emph{BL}}(\linearfamily,\mathbf{p})$ is finite but that best constant is locally bounded. We prove this result in the more general setting of partially localised data. CITATIONS
\begin{theorem}\label{mainthm}
Suppose that $(\linearfamily^0,\mathbf{p})$ is a Brascamp--Lieb datum for which $\emph{BL}(\linearfamily^0,\mathbf{p})<\infty$. Then there exists $\delta>0$ and a constant $C<\infty$ such that
$$\mbox{\emph{BL}}(\linearfamily,\mathbf{p})\leq C$$
whenever $\|\linearfamily-\linearfamily^0\|<\delta$.
\end{theorem}
Of course Theorem \ref{mainthm} tells us that for fixed $\mathbf{p}$, the finiteness set
$$
F(\mathbf{p}):=\{\linearfamily: \mbox{BL}(\linearfamily,\mathbf{p})<\infty\}
$$
is open, and that the function $\linearfamily\mapsto \mbox{BL}(\linearfamily,\mathbf{p})$ is locally bounded. We refer to the concurrent work of Bourgain and Demeter \cite{BD} for some interesting applications of this result in the setting of Weyl sums and Diophantine equations.
%It should be remarked that the mere openness of $F(\mathbf{p})$ is relatively straightforward to verify\footnote{Fix $\linearfamily^0$. For each $V\in\mathcal{V}$ there exists an open set $B(\linearfamily^0,\delta_V)\times B(V,\varepsilon_V)\subseteq \alllinearmaps\times\mathcal{V}$ such that $\dim(V')\leq\sum_{j=1}^mp_j\dim(\linearmapjV')$ whenever $(\linearfamily^0,V')\in B(\linearfamily^0,\delta_V)\times B(V,\varepsilon_V)$ (since $(U,V)\mapsto\dim(U\cap V)$ is lower semi-continuous). Since $\mathcal{V}$ is compact and $\{B(V,\epsilon_V):V\in\mathcal{V}\}$ is an open cover, it has a finite subcover $\{B(V_1,\epsilon_1),\hdots,B(V_N,\epsilon_N)\}$. Setting $\delta=\min\{\delta_{V_1},\hdots,\delta_{V_N}\}$ we conclude that \eqref{char} holds whenever $\|\linearfamily-\linearfamily^0\|<\delta$.} using the compactness of the set $\mathcal{V}$ of subspaces of $H$, along with the elementary fact that the mapping $(U,V)\mapsto\dim(U\cap V)$ is lower semi-continuous on $\mathcal{V}\times\mathcal{V}$. We leave the details of this to the interested reader.

Under certain additional constraints on the kernels of the linear maps $\linearmapj^0 $, the conclusion of Theorem \ref{mainthm} may be seen quite directly. For instance, in the rank one case ($n_j=1$ for all $j$) this follows quickly via Barthe's characterisation of the extreme points of the Brascamp--Lieb polytope
$$
\Pi(\linearfamily):=\{\mathbf{p}:\mbox{BL}(\linearfamily,\mathbf{p})<\infty\},$$
combined with the tautological statement that $\mathbf{p}\in\Pi(\linearfamily)$ if and only if $\linearfamily\in F(\mathbf{p})$; see \cite{Barthe1}. A similar understanding may be reached in the co-rank one case ($n_j=n-1$ for all $j$) using the characterisation of extreme points in Valdimarsson \cite{ValdCJM}. It is also pertinent to note that when the kernels of the maps satisfy the basis condition
\begin{equation}\label{basis}
\bigoplus_{j=1}^m\ker \linearmapj^0 =H,
\end{equation}
a condition which is stable under perturbations of the $\linearmapj $, there is an explicit expression for the Brascamp--Lieb constant $\mbox{BL}(\linearfamily, \mathbf{p})$, from which the conclusion of Theorem \ref{mainthm} (and indeed the \emph{smoothness} of $\linearfamily\mapsto \mbox{BL}(\linearfamily,\mathbf{p})$) is manifest; see \cite{BB}.

If we restrict attention to the so-called \emph{simple} Brascamp--Lieb data, that is, data $(\linearfamily,\mathbf{p})$ for which \eqref{char} holds with strict inequality for all nonzero proper subspaces $V$, much more can be said. In particular, it was shown by Valdimarsson in \cite{ValdBestBest} that the set
$$
F_S(\mathbf{p}):=\{\linearfamily\in F(\mathbf{p}): (\linearfamily,\mathbf{p}) \mbox{ is simple}\}
$$
is open, and that the Brascamp--Lieb constant $\linearfamily\mapsto \mbox{BL}(\linearfamily,\mathbf{p})$ is in fact differentiable there. Since Valdimarsson's argument is based on an application of the implicit function theorem, this regularity conclusion may be pushed even as far as analyticity. However, if $(\linearfamily,\mathbf{p})$ is not simple, that is, there exists a nonzero proper subspace $V$ of $H$ for which \eqref{char} holds with equality (such subspaces are referred to as \emph{critical subspaces}), the situation appears to be much more delicate. In particular, since $F_S(\mathbf{p})$ is open, the mere existence of a critical subspace is unstable under perturbations of $\linearfamily$. This makes a more standard inductive approach to Theorem \ref{mainthm}, via factoring the Brascamp--Lieb constant through critical subspaces, appear quite problematic.
%In Section \ref{examplesec}, we give an example showing that differentiability can fail when there exists $B\in F(\mathbf{p})$ which is not simple.

In this paper we also prove local boundedness for certain \emph{localised} versions of the Brascamp--Lieb constant, including $\mbox{BL}_{\loc}(\linearfamily,\mathbf{p})$, the best constant $C$ in the inequality
\begin{eqnarray}\label{BLloc}
\int_{\|x\|_H\leq 1} \prod_{j=1}^m (f_j \circ \linearmapj )^{p_j}
\leq C \prod_{j=1}^m \bigg(\int_{H_j} f_j\bigg)^{p_j}.
\end{eqnarray}
These inequalities have also been the subject of considerable attention; see \cite{L} and \cite{BCCT1} for a gaussian-localised variant, and the more recent \cite{BCCT2} for a characterisation of finiteness of the best constant.

When $\mathbf{p}$ satisfies the scaling condition \eqref{scaling}, $\mbox{BL}_{\loc}(\linearfamily,\mathbf{p})=\mbox{BL}(\linearfamily,\mathbf{p})$. Thus the stability result Theorem \ref{mainthm} will follow from the corresponding result for $\mbox{BL}_{\loc}(\linearfamily,\mathbf{p})$.

% By H\:older's inequality it is clear that $\mbox{BL}_{\loc}(\linearfamily,\mathbf{p})<\infty$  for a larger family of $\mathbf{p}$ than $\mbox{BL}(\linearfamily,\mathbf{p})$.

Theorem \ref{mainthm} and its local variants (see the forthcoming Theorems \ref{localisedthm} and \ref{partiallylocalisedthm}) are motivated by certain seemingly quite difficult ``perturbed" versions of the Brascamp--Lieb inequality that have arisen in harmonic analysis over the past decade. For such applications we take $H$ and $H_j$ to be $\mathbb{R}^n$ and $\mathbb{R}^{n_j}$, respectively.

The first conjectural generalisation is combinatorial in nature, and takes the form
\begin{equation}\label{combBL}
\int_{\mathbb{R}^n} \prod_{j=1}^m\bigg(\sum_{\alpha_j\in\mathcal{A}_j}f_{j,\alpha_j}\circ L_{j,\alpha_j}\bigg)^{p_j}\leq C
\prod_{j=1}^m\bigg(\sum_{\alpha_j\in\mathcal{A}_j}\int_{\mathbb{R}^{n_j}}f_{j,\alpha_j}\bigg)^{p_j},
\end{equation}
where, for each $1\leq j\leq m$, the linear mappings $(L_{j,\alpha_j})_{\alpha_j\in\mathcal{A}_j}$ are required to be close to a \emph{fixed} surjection $\linearmapj:\mathbb{R}^{n}\rightarrow \mathbb{R}^{n_j}$. Here, for each $1\leq j\leq m$, $\mathcal{A}_j$ indexes the linear maps and arbitrary integrable functions $f_{j,\alpha_j}:\mathbb{R}^{n_j}\rightarrow\mathbb{R}_+$, and the fixed maps
$\linearfamily=(\linearmapj)_{1\leq j\leq m}$ are such that $\mbox{BL}(\linearfamily,\mathbf{p})<\infty$. Such a generalisation is known to hold in some very special cases, the most notable being when the fixed maps $\linearfamily$ and exponents $\mathbf{p}$ correspond to the Loomis--Whitney datum. This is easily seen to be equivalent to the endpoint multilinear Kakeya inequality of Guth \cite{Guth}; see also the non-endpoint versions in \cite{BCT}, \cite{B}, \cite{Guth2} and applications beginning with \cite{BG}. Considering indexing sets $\mathcal{A}_j$, with each consisting of just one element, reveals the statement of Theorem \ref{mainthm} as a necessary feature for such a combinatorial generalisation to hold.

Inequality \eqref{combBL} is best understood via an equivalent formulation obtained by testing it on finite sums of characteristic functions of $\delta$-balls, upon which it may be expressed as
\begin{equation}\label{combBLKak}
\int_{\mathbb{R}^n}\prod_{j=1}^m\Bigl(\sum_{T_j\in\mathbb{T}_j}\chi_{T_j}\Bigr)^{p_j}\leq C\delta^n\prod_{j=1}^m(\#\mathbb{T}_j)^{p_j},
\end{equation}
uniformly in $\delta$, where for each $j$, $\mathbb{T}_j$ denotes an arbitrary finite collection of $\delta$-neighbourhoods of $n_j'$-dimensional affine subspaces of $\mathbb{R}^n$ which, modulo translations, are close to the fixed subspace $V_j:=\ker L_j$. Here $
n_j':=n-n_j
$, and we use the standard metric on the Grassmann manifold of $n_j'$-dimensional subspaces of $\mathbb{R}^n$.
Notice that the characterisation of finiteness of $\mbox{BL}(\linearfamily^0,\mathbf{p})$, given by \eqref{scaling} and \eqref{char}, depends only on the \emph{kernels} of the linear maps $\linearmapj$. In particular, for $V_j:=\ker L_j$, the condition \eqref{char} may be rewritten as $$\dim(V)\leq\sum_{j=1}^mp_j\dim(V\cap V_j^\perp).$$

A second generalisation of the Brascamp--Lieb inequality is oscillatory in nature, and belongs to the restriction theory of the Fourier transform. To describe this suppose that, for each $1\leq j\leq m$, $\Sigma_j:U_j\rightarrow\mathbb{R}^n$ is a smooth parametrisation of a $n_j$-dimensional submanifold $S_j$ of $\mathbb{R}^n$ by a neighbourhood $U_j$ of the origin in $\mathbb{R}^{n_j}$. We associate to each $\Sigma_j$ the \textit{extension operator}
$$
E_jg_j(\xi):=\int_{U_j}e^{2\pi i\xi\cdot\Sigma_j(x)}g_j(x)dx,
$$
where $\xi\in\mathbb{R}^n$. In this setting it is natural to conjecture that if $\mbox{BL}(\linearfamily,\mathbf{p})<\infty$, where $\linearmapj:=(\mathrm{d}\Sigma_j(0))^*$ for each $j$, then provided the neighbourhoods $U_j$ of $0$ are chosen small enough, the inequality
\begin{equation}\label{genmultrest}
\int_{\mathbb{R}^n}\prod_{j=1}^m|E_jg_j|^{2p_j}\leq C\prod_{j=1}^m\|g_j\|_{L^2(U_j)}^{2p_j}
\end{equation}
holds for all $g_j\in L^2(U_j)$, $1\leq j\leq m$. The weaker inequality
\begin{equation*}%\label{genmultrestlosspre}
\int_{B(0,R)}\prod_{j=1}^m|E_jg_j|^{2p_j}\leq C_\varepsilon R^\varepsilon\prod_{j=1}^m\|g_j\|_{L^2(U_j)}^{2p_j},
\end{equation*}
involving an arbitrary $\varepsilon>0$ loss was established in the particular case when $(\linearfamily, \mathbf{p})$ is the Loomis--Whitney datum in \cite{BCT}, and has had extensive applications and developments beginning with \cite{BG}; see also \cite{Bourgaintrig}, \cite{Bourgainmax}, \cite{BSSY}, \cite{BDAnnals}, \cite{BourgainDER}, \cite{BourSome}, \cite{BDgeneral}, \cite{BourgainWatt}, \cite{BD}, \cite{LV}. The endpoint \eqref{genmultrest} is only known in very elementary situations, and is easily seen to be best possible in the sense that $\mbox{BL}(\linearfamily,\mathbf{p})<\infty$ provides a necessary condition on the $p_j$ by taking \textit{linear} $\Sigma_j$; see \cite{B} for further discussion.

%such as when $n_j=n$ for all $j$ (whence it reduces to the multilinear H\"older inequality), or $m=n=2$, $n_1=n_2=1$, $p_1=p_2=1$, which becomes the well-known bilinear restriction inequality in the plane;

A third, seemingly more modest generalisation of the Brascamp--Lieb inequality, originating in \cite{BCW}, involves dropping the linearity requirement on the maps $\linearmapj $, and instead considering $\nonlinearmapj$ smooth
submersions in a neighbourhood of a point $x_0 \in \mathbb{R}^n$. In this context it seems natural to conjecture (see \cite{BB}) that provided
$$
\mbox{BL}_{\loc}(\mathrm{d} \textbf{\mbox{B}}(x_0), \textbf{\mbox{p}}) < \infty,
$$
there exists a neighbourhood $U$ of $x_0$ and a finite constant $C$
such that
\begin{equation}\label{localst}
\int_U \prod_{j=1}^m (f_j \circ \nonlinearmapj)^{p_j} \leq C \prod_{j=1}^m \bigg(\int_{\mathbb{R}^{n_j}} f_j  \bigg)^{p_j},
\end{equation}
or equivalently
\begin{equation}\label{localstform}
\int_U \prod_{j=1}^m f_j \circ \nonlinearmapj \leq C \prod_{j=1}^m \|f_j\|_{L^{q_j}(\mathbb{R}^{n_j})},
\end{equation}
where, as in \eqref{BLeq}, $q_j=p_j^{-1}$.
Here $\mathrm{d} \textbf{\mbox{B}}(x_0)=(\mathrm{d}\nonlinearmapj(x_0))$,
where $\mathrm{d}\nonlinearmapj(x_0):\mathbb{R}^n\rightarrow\mathbb{R}^{n_j}$
denotes the derivative of $\nonlinearmapj$ at the point $x_0$. Such a
generalisation has been shown to hold under the basis condition \eqref{basis} on the derivative maps
$\mathrm{d}\nonlinearmapj(x_0)$; we refer to \cite{BB}, \cite{BCW}, \cite{BHT}, \cite{BH} for this and applications to
problems in euclidean harmonic analysis and dispersive PDE. An elementary scaling and limiting argument shows that if \eqref{localst} holds then there exists a neighbourhood $U'$ of $x_0$ such that $$\sup_{x\in U'}\mbox{BL}_{\loc}(\mathrm{d} \textbf{\mbox{B}}(x), \textbf{\mbox{p}}) < \infty,$$ a statement which is closely related to the local boundedness of the (linear) localised Brascamp--Lieb constant; see \cite{BBG1} for further details. The local variant of Theorem \ref{mainthm} (see Theorem \ref{localisedthm}) may thus be viewed as a modest first step towards the general form of this nonlinear Brascamp--Lieb conjecture.

Our applications of Theorem \ref{mainthm} consist of proving certain weak forms of the generalised Brascamp--Lieb inequalities \eqref{combBLKak}, \eqref{genmultrest} and \eqref{localst}, where one accepts some arbitrarily small loss in regularity of the input functions. All of these combine our stability results with well-known variants of the induction-on-scales method.

Our application to the variant \eqref{combBL} is best expressed in terms of the equivalent geometric formulation \eqref{combBLKak}, and is the following.
\begin{theorem}\label{CBLapp}
Suppose $(\linearfamily,\mathbf{p})$ is a Brascamp--Lieb datum for which $\mbox{\emph{BL}}(\linearfamily,\mathbf{p})<\infty$. Then there exists $\nu>0$ such that for every $\varepsilon>0$,
\begin{equation}\label{combBLweak}
\int_{[-1,1]^n}\prod_{j=1}^m\Bigl(\sum_{T_j\in\mathbb{T}_j}\chi_{T_j}\Bigr)^{p_j} \leq C_\varepsilon\delta^{n-\varepsilon}
\prod_{j=1}^m\left(\#\mathbb{T}_j\right)^{p_j}
\end{equation}
holds for all finite collections $\mathbb{T}_j$ of $\delta$-neighbourhoods of $n_j'$-dimensional affine subspaces of $\mathbb{R}^n$ which, modulo translations, are within a distance $\nu$ of the fixed subspace $V_j:=\ker L_j$.
\end{theorem}
In the particular case of \emph{gaussian-extremisable} Brascamp--Lieb data $(\linearfamily,\mathbf{p})$, the above theorem may be seen as a consequence of Corollary 4.2 in \cite{BCT}; see \cite{BCCT1} for a characterisation of such data.
As in the case where $(\linearfamily,\mathbf{p})$ is the Loomis--Whitney datum (see \cite{BCT}), the above theorem implies the following very general restriction theorem.  We refer to \cite{BD} for recent number-theoretic applications of these ``multilinear" restriction and Kakeya-type inequalities.
\begin{theorem}\label{OBLapp}
Suppose that $\mbox{\emph{BL}}(\linearfamily,\mathbf{p})<\infty$, where $\linearmapj:=(\mathrm{d}\Sigma_j(0))^*$ for each $j$. Then there exist neighbourhoods $U_j$ of $0\in\mathbb{R}^{n_j}$, $1\leq j\leq m$, such that for every $\varepsilon>0$,
\begin{equation}\label{genmultrestloss}
\int_{B(0,R)}\prod_{j=1}^m|E_jg_j|^{2p_j}\leq C_\varepsilon R^\varepsilon\prod_{j=1}^m\|g_j\|_{L^2(U_j)}^{2p_j}
\end{equation}
holds for all $g_j\in L^2(U_j)$, $1\leq j\leq m$, and all $R \geq 1$.
\end{theorem}
Our application to the variant \eqref{localst} is best expressed in terms of the equivalent formulation \eqref{localstform}, and involves a regularity loss that may be captured in the scale of the classical Sobolev spaces. (Absorbing this loss in the scale of $L^p$ spaces appears to be less clear.)
\begin{theorem}\label{NLBLapp}
Suppose $(\linearfamily,\mathbf{p})$ is a  Brascamp--Lieb datum for which $\mbox{\emph{BL}}(\linearfamily,\mathbf{p})<\infty$, and that for each $1\leq j\leq m$, $\nonlinearmapj:\mathbb{R}^{n}\rightarrow\mathbb{R}^{n_j}$ is a smooth submersion in a neighbourhood of the origin satisfying $\mathrm{d}\nonlinearmapj(0)=\linearmapj $. Then there exists a neighbourhood $U$ of the origin in $\mathbb{R}^n$ such that for every $\varepsilon>0$,
\begin{equation}
\int_U\prod_{j=1}^m f_j\circ \nonlinearmapj\leq C_\varepsilon \prod_{j=1}^m\|f_j\|_{L^{q_j}_\varepsilon(\mathbb{R}^{n_j})},
\end{equation}
where $q_j = p_j^{-1}$.
\end{theorem}
Here, we use the notation $\|f\|_{L^q_\varepsilon(\mathbb{R}^n)} = \| (1 - \Delta)^{\varepsilon/2}f\|_{L^q(\mathbb{R}^n)}$ for $q \geq 1$ and $n \in \mathbb{N}$. It is worth noting that the proof allows the smoothness condition on the $B_j$ to be relaxed to $C^{1,\beta}$ for any $\beta>0$.

\emph{Structure of the paper.} In Section \ref{stabSection} we prove Theorem \ref{mainthm} via the corresponding statement for the local Brascamp--Lieb constant $\mbox{BL}_{\loc}(\linearfamily,\mathbf{p})$. We conclude Section \ref{stabSection} by unifying these results in the setting of partially-localised Brascamp--Lieb constants. We prove Theorems \ref{CBLapp}--\ref{NLBLapp} in Section \ref{Sectionappl}.

\emph{Acknowledgement.} We thank Ciprian Demeter for clarifying the
connections between this work and the concurrent preprint \cite{BD}.

\section{Stability of the Brascamp--Lieb constant}\label{stabSection}
\subsection{Openness}\label{opennesssection}
Although the proof of the local boundedness of $\linearfamily\mapsto \mbox{BL}(\linearfamily,\mathbf{p})$ simultaneously establishes the openness of $F(\mathbf{p})$, the latter permits a much more elementary approach. We describe this first.

It suffices to prove that if the dimension condition \eqref{char} holds for $\linearfamily=\linearfamily^0$ then there exists $\delta>0$ such that \eqref{char} holds whenever $\|\linearfamily-\linearfamily^0\|<\delta$. For each $1\leq k\leq n$ let $\mathcal{E}_k$ denote the compact set of all orthonormal sets $\mathbf{e}:=\{e_1,\hdots,e_k\}$ in $H$. This notation allows us to rewrite \eqref{char} as
\begin{equation}\label{char1}
k\leq\sum_{j=1}^m p_j\dim(\langle \linearmapj e_1,\hdots,\linearmapj e_k\rangle)
\end{equation}
for all $\mathbf{e}\in\mathcal{E}_k$ and $1\leq k\leq n$.

Fix $k$ and let $\mathbf{e}\in\mathcal{E}_k$. Since \eqref{char1} holds with $\linearfamily=\linearfamily^0$, for each $1\leq j\leq m$ we may choose a subset $I_j\subseteq\{1,\hdots,k\}$ satisfying $|I_j|=\dim(\langle \linearmapj^0e_1,\hdots,\linearmapj^0e_k\rangle)$,
\begin{equation}\label{link}
k\leq \sum_{j=1}^mp_j|I_j|
\end{equation}
and
$$\bigwedge_{i\in I_j} \linearmapj^0e_i\not=0.$$

Since
$$
(\linearfamily,\mathbf{e}')\mapsto\bigwedge_{i\in I_j} \linearmapj e_i'\in\Lambda^{|I_j|}(H_j)
$$
is continuous for each $j$, there exist $\varepsilon(\mathbf{e}), \delta(\mathbf{e})>0$ such that
$$\bigwedge_{i\in I_j} \linearmapj e_i'\not=0$$
for each $j$, whenever $\|\mathbf{e}'-\mathbf{e}\|<\varepsilon(\mathbf{e})$ and $\|\linearfamily-\linearfamily^0\|<\delta(\mathbf{e})$.
In particular, $\dim(\langle \linearmapj e_1',\hdots,\linearmapj e_k'\rangle)\geq |I_j|$ for each $j$, and so by \eqref{link},
$$
k\leq\sum_{j=1}^m p_j\dim(\langle \linearmapj e_1',\hdots,\linearmapj e_k'\rangle)
$$
whenever $\|\mathbf{e}'-\mathbf{e}\|<\varepsilon(\mathbf{e})$ and $\|\linearfamily-\linearfamily^0\|<\delta(\mathbf{e})$.
Since $\mathcal{E}_k$ is compact there exists a finite collection $\mathbf{e}^1,\hdots,\mathbf{e}^N\in\mathcal{E}_k$ such that the sets
$$
\{\mathbf{e}'\in\mathcal{E}_k:\|\mathbf{e}'-\mathbf{e}^\ell\|<\varepsilon(\mathbf{e}^\ell)\},
$$
with $\ell=1,\hdots,N$, cover $\mathcal{E}_k$. Finally, choosing $\delta=\min\{\delta(\mathbf{e}^1),\hdots,\delta(\mathbf{e}^N)\}$ we conclude that \eqref{char1} holds whenever $\|\linearfamily-\linearfamily^0\|<\delta$ and $\mathbf{e}\in\mathcal{E}_k$. Since there are boundedly many such $k$, the claimed openness follows.

\subsection{Local boundedness for localised data}
In this section we prove the following local version of Theorem \ref{mainthm}:
\begin{theorem}\label{localisedthm}
Suppose that $(\linearfamily^0,\mathbf{p})$ is a Brascamp--Lieb datum  such that  $\emph{BL}_{\loc}(\linearfamily^0,\mathbf{p})<\infty$. Then there exists $\delta>0$ and a constant $C<\infty$ such that
$$\mbox{\emph{BL}}_{\loc}(\linearfamily,\mathbf{p})\leq C$$
whenever $\|\linearfamily-\linearfamily^0\|<\delta$.
\end{theorem}

In \cite{BCCT1} it is shown that $\mathrm{BL}_{\loc}(\linearfamily^0,\mathbf{p})$ is finite if and only if
% % \begin{equation}\label{dimensionpl}
% % \dim(V)\leq \sum_{j=1}^m p_j \; \dim(\linearmapj^0(V)) \;
% % \text{ for all subspaces $V\subseteq H_0$}
% % \end{equation}
% % and
\begin{equation}\label{codimensionloc}
\codim_H(V)
\geq
\sum_{j=1}^m p_j \codim_{H_j}(\linearmapj^0 V) \quad
\text{ for all subspaces $V\subset H$.}
\end{equation}
% Further, this characterisation of finiteness holds for $\mbox{BL}_{\loc}(\linearfamily,\mathbf{p})$ as well \cite{BCCT2}.

Note that Theorem \ref{mainthm} is a direct corollary of Theorem \ref{localisedthm} in the case where the condition \eqref{scaling} is satisfied by a scaling argument.

Our proof of Theorem \ref{localisedthm} amounts to an appropriately uniform version of the proof of the finiteness characterisation theorem for the Gaussian localised version in \cite{BCCT1}. The advantage of this approach over the alternative in \cite{BCCT2} is that it avoids reference to critical subspaces, objects whose existence is unstable under perturbations of $\linearfamily$. Our argument fails to yield a more quantitative statement, such as something closer to upper semi-continuity for the Brascamp--Lieb constant, due to the crucial role played by the compactness of appropriately nondegenerate bases for $H$.

As in the proof of the openness of $F(\mathbf{p})$ given in Section \ref{opennesssection}, we shall exploit the finiteness condition \eqref{codimensionloc} through the consideration of an appropriate set of bases of $H$. The key tool is a uniform version of Lemma 5.1 from \cite{BCCT1}.

\begin{lemma}\label{loc51} Suppose  $(\linearfamily^0,\mathbf{p})$ is a Brascamp--Lieb datum such that  $\emph{BL}_{\loc}(\linearfamily^0,\mathbf{p})<\infty$.
Then there exist real numbers $c,\delta > 0$ such that for every $\mathbf{e}\in \mathcal{E}_n$ and every $\linearfamily$ satisfying $\|\linearfamily-\linearfamily^0\|\leq\delta$, there exists a set $I_j \subseteq \{1,\ldots,n\}$ with $|I_j| = \dim(H_j)$ for each $1\leq j\leq m$, such that
\begin{equation}\label{loc35b}
\sum_{j=1}^m p_j |I_j \cap \{1,\ldots,k\}| \leq k \hbox{ for all } 1 \leq k \leq n,
\end{equation}
and
\begin{equation}\label{sizebloc}
\bigg| \bigwedge_{i \in I_j} \linearmapj  e_i \bigg| \geq c\hbox{ for all } 1 \leq j \leq m.
\end{equation}
% Here $\bigwedge_{i \in I_j} \linearmapj  e_i$ denotes the wedge product of the vectors $\linearmapj  e_i$, and the $\| \|_H$ norm is the usual norm on forms induced from the Hilbert space structure.
\end{lemma}
In the above lemma, and throughout, we identify $\bigwedge_{i \in I_j} \linearmapj  e_i$ with a real number via Hodge duality.
\begin{proof}
Let $\mathcal{I}$ denote the set of all $m$-tuples $(I_1,\hdots,I_m)$ of subsets of $\{1,\hdots, n\}$ satisfying $|I_j| = \dim(H_j)$ and \eqref{loc35b}. Define % the function $h:\alllinearmaps\times\mathcal{E}_n\rightarrow [0,\infty)$ by
$$
h(\linearfamily,\mathbf{e})=\max_{(I_1,\hdots,I_m)\in\mathcal{I}}\min_{1\leq j\leq m}\bigg|\bigwedge_{i\in I_j}\linearmapj e_i\bigg|.
$$
We begin by proving that
$h(\linearfamily^0,\mathbf{e})\geq c' $ for all $\mathbf{e}\in\mathcal{E}_n$ and some $c'>0$. By the continuity of $h$ and the compactness of $\mathcal{E}_n$, it is enough to verify that $h(\linearfamily^0,\mathbf{e})\neq 0$ for all $\mathbf{e}\in\mathcal{E}_n$. From the definition of $h$, it suffices to show that there exists $(I_1,\hdots,I_m)\in\mathcal{I}$ for which
\begin{equation}\label{want1loc}\bigwedge_{i \in I_j} \linearmapj^0 e_i  \neq 0 \hbox{ for all } 1 \leq j \leq m.
\end{equation}
Proceeding as in \cite{BCCT1}, we fix $j$ and select $I_j$ by a backwards greedy algorithm, firstly by putting $i_0$ in $I_j$, where
\[
i_0 = \max\{ i \in \{1,\ldots,n\} : L_j^0e_i \neq 0\}
\]
and then choosing indices $i \in \{1,\ldots,i_0-1\}$ for which $\linearmapj^0 e_i$ is not in the linear span of $\{ \linearmapj^0 e_{i'}: i < i' \leq n\}$. By construction \eqref{want1loc} holds, and since $\linearmapj^0$ is surjective, $|I_j|=\dim (H_j)$.

To prove \eqref{loc35b}, we apply the codimension condition
\eqref{codimensionloc} with $V$ equal to the span of $\{e_{k+1},\ldots,e_n\}$, to obtain
$$ \sum_j p_j \dim(H_j / \linearmapj^0 V) \leq k.$$
By construction of $I_j$,  $\dim(\linearmapj^0 V) = |I_j \cap \{k+1,\ldots,n\}|$
and hence $\dim(H_j / \linearmapj^0 V) = |I_j \cap \{1,\ldots,k\}|$.
%Thus $(I_1,...,I_m)\in\mathcal{I}$ and satisfies \eqref{want1loc}, as required.

Now let $\mathcal{K}$ be a compact set of linear maps with $\linearfamily^0$ belonging to its interior.  Since the function $h$ is uniformly continuous on the compact set $\mathcal{K}\times\mathcal{E}_n$, there exists $\delta>0$ such that
$$
|h(\linearfamily,\mathbf{e})-h(\linearfamily^0,\mathbf{e})|\leq \frac{c'}{2}
$$
and $\linearfamily\in\mathcal{K}$ whenever $\mathbf{e}\in\mathcal{E}_n$ and $\|\linearfamily-\linearfamily^0\|\leq\delta$. Therefore,
$h(\linearfamily,\mathbf{e})\geq\frac{c'}{2}$ whenever $\mathbf{e}\in\mathcal{E}_n$ and $\|\linearfamily-\linearfamily^0\|\leq\delta$.
The lemma now follows from the definition of $h$.% and the fact that the cardinality of $\mathcal{I}$ is a constant depending only on $n$.
\end{proof}

\begin{proof}[Proof of Theorem \ref{localisedthm}]
We assume, as we may, that $p_j>0$ and $H_j\neq\{0\}$ for each $j$. Let  $c$ and  $ \delta$ be those given by Lemma \ref{loc51}. We emphasise that these quantities depend only on the fixed datum $(\linearfamily^0,\mathbf{p})$. To further emphasise uniformity we include the explicit constant factors arising in the remainder of the argument.

The constant $\mbox{BL}_{\loc}(\linearfamily,\mathbf{p})$ %, in
%  \begin{eqnarray}\label{BLloc}
% \int_{\|x\|_H\leq 1} \prod_{j=1}^m (f_j \circ \linearmapj )^{p_j}
% \leq C \prod_{j=1}^m \left(\int_{H_j} f_j\right)^{p_j}.
% \end{eqnarray}
is bounded above by a fixed multiple of the best constant in the Gaussian localised case,
\[
 \int_{H} e^{-\pi|\cdot|^2} \prod_{j=1}^m (f_j \circ \linearmapj )^{p_j}
 \leq C \prod_{j=1}^m \Bigl(\int_{H_j} f_j\Bigr)^{p_j}.
\]
By an application of Lieb's Theorem (Theorem 6.2 in \cite{L}), we have
$$
\mbox{BL}_{\loc}(\linearfamily,\mathbf{p})\leq C\sup\;\frac{\prod_{j=1}^m(\det A_j)^{p_j/2}}{\det(M+\id)^{1/2}},
$$
where $M=\sum_{j=1}^m p_j\linearmapj^*A_j\linearmapj$, $\id$ is the identity matrix, and the supremum is taken over all positive definite $A_j:H_j\rightarrow H_j$, $1\leq j\leq m$.
It will thus suffice to prove that there exists a constant $C>0$ such that
\begin{equation} \label{e:afterlieb}
\prod_{j=1}^m (\det A_j)^{p_j} \leq C \det (M+\id)
\end{equation}
for all data $(\linearfamily,\mathbf{p})$ such that $\|\linearfamily-\linearfamily^0\|\leq\delta$ and all such positive definite $A_j$.

Since $p_j>0$ and $\bigcap_{j=1}^m \ker{\linearmapj}=\{0\}$, we have that $M$ and $M+\id$ are positive definite. Let $e_1,\hdots,e_n$ be an orthonormal basis of eigenvectors for $M+\id$, ordered so that their corresponding eigenvalues satisfy $\mu_1 \geq \ldots \geq \mu_n > 1$.

For each $1 \leq i \leq n$ we have that
$$ \langle e_i, Me_i \rangle_{H}=\mu_i-\langle e_i, e_i \rangle_{H}\leq \mu_i,
$$
and so for each $1 \leq i \leq n$ and $1 \leq j \leq m$,
$$ \langle A_j \linearmapj e_i, \linearmapj e_i \rangle_{H_j} = \langle e_i, \linearmapj^* A_j \linearmapj e_i \rangle_{H} \leq \frac{1}{p_j} \langle e_i, Me_i \rangle_{H} \leq  \mu_i / p_j.$$

Applying Lemma \ref{loc51}, for each $1 \leq j \leq m$, there exists $I_j \subseteq \{1,\ldots,n\}$  of cardinality $|I_j| = \dim(H_j)$
 such that \eqref{loc35b} and \eqref{sizebloc} hold.
For fixed $1 \leq j \leq m$, if we consider $\linearmapj^*A_j\linearmapj$ acting on the subspace spanned by $\{ e_i : i \in I_j\}$, then, since the determinant of a positive semi-definite transformation is at most the product of its diagonal entries,
$$
\det(A_j) \leq \bigg|  \bigwedge_{i \in I_j} \linearmapj e_i \bigg|^{-2}  \prod_{i \in I_j} \langle \linearmapj^*A_j\linearmapj  e_i,  e_i \rangle_{H}.
$$
Thus
$$
\det(A_j) \leq \left(c^2p_j^{n_j}\right)^{-1} \prod_{i \in I_j} \mu_i,
$$
where $c>0$ is the constant given by \eqref{sizebloc}, and this implies
$$
\prod_{j=1}^m (\det A_j)^{p_j} \leq\bigg(c^{2\sum_{j=1}^mp_j}\prod_{j=1}^m p_j^{p_jn_j}\bigg)^{-1} \prod_{i=1}^n \mu_i^{a_i},
$$
where $a_i := \sum_{j=1}^m p_j |I_j \cap \{i\}|$. By telescoping we may write
$$
\prod_{i=1}^n \mu_i^{a_i} = \det(M+\id)\prod_{k=1}^n \left(\frac{\mu_{k+1}}{\mu_k}\right)^{k - \sum_{i = 1}^k a_i}
$$
since $\det(M+\id) = \prod_{i=1}^n \mu_i$, and where we have defined $\mu_{n+1} := 1$.
%$$
%\prod_{i=1}^n \mu_i^{\sum_{j=1}^m p_j |I_j \cap \{i\}|} = \left(\prod_{i=1}^n \mu_i\right) \prod_{k=1}^n \left(\frac{\mu_{k+1}}{\mu_k}\right)^{k - a_1 - \ldots - a_k},
%$$
%where we define $\mu_{n+1} = 1$.
Applying \eqref{loc35b}, $k - \sum_{i = 1}^k a_i \geq 0$ and, by construction, $\frac{\mu_{k+1}}{\mu_k} \leq 1$ for all $1 \leq k \leq n$. Hence \eqref{e:afterlieb} holds with constant $C = (c^{2\sum_{j=1}^mp_j}\prod_{j=1}^m p_j^{p_jn_j})^{-1}.$
\end{proof}

\subsection{Local boundedness for partially localised data }
In this section we prove a generalisation of Theorem \ref{mainthm} for partially localised Brascamp--Lieb constants (see \cite{L} and more recently \cite{BCCT2}).

Let $(\linearfamily,\mathbf{p})$ be a Brascamp--Lieb datum,  let $H_0\subseteq H$ be a subspace of $H$, and let $G$ be a positive semi-definite linear map whose kernel is $H_0$. The associated partially localised Brascamp--Lieb inequality is
\begin{eqnarray}
 \int_{\{x\in H: |\langle Gx,x \rangle|<1\}}  \prod_{j=1}^m (f_j \circ \linearmapj)^{p_j}
 \leq C \prod_{j=1}^m \bigg(\int_{H_j} f_j \bigg)^{p_j}.
\end{eqnarray}
Denote the best constant in the above inequality by $\mbox{BL}_G(\linearfamily,\mathbf{p})$ (not to be confused with $\mbox{BL}_{\bf g}$ from \cite{BCCT1}).

In \cite{BCCT2} it is shown that $\mbox{BL}_G(\linearfamily,\mathbf{p})$ is finite if and only if
\begin{equation}\label{dimensionpl}
\dim(V)\leq \sum_{j=1}^m p_j  \dim(\linearmapj V) \quad
\text{ for all subspaces $V\subseteq H_0$}
\end{equation}
and
\begin{equation}\label{codimensionpl}
\codim_H(V)
\geq
\sum_{j=1}^m p_j \codim_{H_j}(\linearmapj V) \quad
\text{ for all subspaces $V\subset H$.}
\end{equation}

It is tempting to believe that the partially localised case should follow easily from the localised case, Theorem \ref{localisedthm}, by a scaling argument as Theorem \ref{mainthm} does.  However, the scaling argument in the partially localised setting requires an anisotropic dilation, which changes the initial Brascamp--Lieb datum nontrivially. Nevertheless, a version of Theorem \ref{mainthm} holds in this case as well. Again our proof is an appropriately uniform version of the finiteness characterisation in \cite{BCCT1} combining the methods from the fully-local and fully-global cases. This gives a proof of the characterisation of finiteness for partially localised data which does not require factoring through critical subspaces as in \cite{BCCT2}.
\begin{theorem}\label{partiallylocalisedthm}
Suppose that $(\linearfamily^0,\mathbf{p})$ is a Brascamp--Lieb datum and $G$ is a positive semi-definite linear map such that $\emph{BL}_G(\linearfamily^0,\mathbf{p})<\infty$. Then there exists $\delta>0$ and a constant $C<\infty$ such that
$$\mbox{\emph{BL}}_G(\linearfamily,\mathbf{p})\leq C$$
whenever $\|\linearfamily-\linearfamily^0\|<\delta$.
\end{theorem}

As in the proof of Theorem \ref{localisedthm}, we shall exploit the conditions \eqref{dimensionpl} and \eqref{codimensionpl} through the consideration of an appropriate set of bases of $H$. However, rather than using orthonormal bases, it will be important to use classes which have some alignment with the distinguished subspace $\ker G = H_0$, and to permit bases which are not quite orthonormal.

For $0<\alpha\leq 1$ let $\mathcal{V}_\alpha$ denote the set of all $\mathbf{v}=(v_{1},\hdots,v_n)\in H^n$ such that
$\|v_i\|\leq1 $ for all $1\leq i \leq n$, and
$$
\bigg|  \bigwedge_{i=1}^n  v_i \bigg| \geq \alpha.
$$
Further, for each $\ell\in\mathbb{N}$ with $n-\dim H_0\leq\ell< n$, let
$$
\mathcal{V}_{\alpha,\ell}:=\{\mathbf{v}\in\mathcal{V}_\alpha:v_{\ell+1},\cdots,v_n\in H_0\},
$$
and $\mathcal{V}_{\alpha,n}:=\mathcal{V}_\alpha$. We thus interpret an element $\mathbf{v}$ of $\mathcal{V}_{\alpha,\ell}$ as a certain (ordered) basis for $H$ with a lower bound on its degeneracy.

Clearly $\mathcal{V}_{\alpha,\ell}\subseteq\mathcal{V}_{\alpha,\ell+1}$ for each $\ell$. Note also that $\mathcal{V}_{\alpha,\ell}$ is compact for each $\alpha$ and $\ell$.
% Note that $\mathcal{N}_{G,\alpha}$ is compact as the kernel of ${G}$ is closed.

\begin{lemma}\label{new51} Suppose that $n-\dim H_0\leq\ell\leq n$ and $\alpha\in (0,1]$, and that $(\linearfamily^0,\mathbf{p})$ is a Brascamp--Lieb datum for which $\emph{BL}_G(\linearfamily^0,\mathbf{p})<\infty$. Then there exist real numbers $c_\ell,\delta_\ell > 0$ such that for every $\mathbf{v}\in \mathcal{V}_{\alpha,\ell}$ and every $\linearfamily$ satisfying $\|\linearfamily-\linearfamily^0\|\leq\delta_\ell$, there exists a set $I_j \subseteq \{1,\ldots,n\}$ with $|I_j| = \dim(H_j)$ for each $1\leq j\leq m$, such that
\begin{equation}\label{35b}
\sum_{j=1}^m p_j |I_j \cap \{1,\ldots,k\}| \leq k \hbox{ for all } 0 \leq k \leq n,
\end{equation}
\begin{equation}\label{35ker}
\sum_{j=1}^m p_j |I_j \cap \{k+1,\ldots,n\}| \geq n - k \hbox{ for all } \ell \leq k \leq n,
\end{equation}
and
\begin{equation}\label{sizeb}
\bigg| \bigwedge_{i \in I_j} \linearmapj v_i \bigg| \geq c_\ell \hbox{ for all } 1 \leq j \leq m.
\end{equation}
% Here $\bigwedge_{i \in I_j} \linearmapj e_i$ denotes the wedge product of the vectors $\linearmapj e_i$, and the $\| \|_H$ norm is the usual norm on forms induced from the Hilbert space structure.
\end{lemma}
\begin{proof}
Let $\mathcal{I}_\ell$ denote the set of all $m$-tuples $(I_1,\hdots,I_m)$ of subsets of $\{1,\hdots, n\}$ satisfying $|I_j| = \dim(H_j)$, \eqref{35b}, and \eqref{35ker}. Define %the function $h_\ell:\alllinearmaps\times\mathcal{V}_{\alpha,\ell}\rightarrow [0,\infty)$ by
$$
h_\ell(\linearfamily,\mathbf{v})=\max_{(I_1,\hdots,I_m)\in\mathcal{I}_\ell}\min_{1\leq j\leq m} \bigg| \bigwedge_{i\in I_j}\linearmapj v_i \bigg|.
$$
We begin by proving that
$h_\ell(\linearfamily^0,\mathbf{v})\geq c_{\ell}'$ for all $\mathbf{v}\in\mathcal{V}_{\alpha,\ell}$ and some $c_\ell'>0$. By the continuity of $h_\ell$ and the compactness of $\mathcal{V}_{\alpha,\ell}$, it is enough to verify that $h_\ell(\linearfamily^0,\mathbf{v})\neq 0$ for all $\mathbf{v}\in\mathcal{V}_{\alpha,\ell}$. From the definition of $h_\ell$ it suffices to show that there exists $(I_1,\hdots,I_m)\in\mathcal{I}_\ell$ for which
\begin{equation}\label{want1}
\bigwedge_{i \in I_j} \linearmapj^0 v_i  \neq 0 \hbox{ for all } 1 \leq j \leq m.
\end{equation}
Again, we select each $I_j$ by a backwards greedy algorithm, and \eqref{35b} follows as before.  % taking $I_j$ to be equal to those indices $i$ for which
% $\linearmapj^0 e_i$ is not in the linear span of $\{ \linearmapj^0 v_{i'}: i < i' \leq n\}$. By construction \eqref{want1} holds, and since $\linearmapj^0$ is surjective, $|I_j|=\dim (H_j)$.
% To prove , we apply the hypothesis
% \eqref{codimensionpl} with $V$ equal to the span of $\{v_{k+1},\ldots,v_n\}$, to obtain
% $$ \sum_j p_j \dim(H_j / \linearmapj^0 V) \leq k.$$
% By construction of $I_j$,  $\dim(\linearmapj^0 V) = |I_j \cap \{k+1,\ldots,n\}|$
% and hence $\dim(H_j / \linearmapj^0 V) = |I_j \cap \{1,\ldots,k\}|$.
To prove \eqref{35ker}, we let $\ell \leq k \leq n$ and apply \eqref{dimensionpl}, which is a consequence of our hypothesis that $\mathrm{BL}_G(\linearfamily^0,\mathbf{p})<\infty$, with $V$ equal to the span of $\{v_{k+1},\ldots,v_n\} \subset H_0$ to obtain
$$
\sum_{j=1}^m p_j \dim(\linearmapj^0 V) \geq n-k.
$$
By construction of $I_j$,  we have $\dim(\linearmapj^0 V) = |I_j \cap \{k+1,\ldots,n\}|$.
Thus $(I_1,...,I_m)\in\mathcal{I}_\ell$ satisfies \eqref{want1}, as required.

Now let $\mathcal{K}$ be a compact set of linear maps which contains $\linearfamily^0$ in its interior. Since the function $h_\ell$ is uniformly continuous on the compact set $\mathcal{K}\times\mathcal{V}_{\ell,\alpha}$, there exists $\delta_\ell>0$ such that
$$
|h_\ell(\linearfamily,\mathbf{v})-h_\ell(\linearfamily^0,\mathbf{v})|\leq \frac{c_\ell'}{2}
$$
and $\linearfamily\in\mathcal{K}$ whenever $\mathbf{v}\in\mathcal{V}_{\alpha,\ell}$ and $\|\linearfamily-\linearfamily^0\|\leq\delta_\ell$. Therefore,
$h_\ell(\linearfamily,\mathbf{v})\geq\frac{c_\ell'}{2}$ whenever $\mathbf{v}\in\mathcal{V}_{\alpha,\ell}$ and $\|\linearfamily-\linearfamily^0\|\leq\delta_\ell$.
The lemma now follows from the definition of $h_\ell$. %and the fact that the cardinality of $\mathcal{I}_\ell$ is at most a constant depending only on $n$.
\end{proof}

\begin{proof}[Proof of Theorem \ref{partiallylocalisedthm}]
We assume, as we may, that $p_j>0$ and $H_j\neq\{0\}$ for each $j$. By applying a linear transformation we may also assume that $G$ is the orthogonal projection of $H$ onto $H_0^\perp$. We may also reduce to the case where $n\geq 2$ as when $n=1$, $G$ is either the identity or $0$. % case where $G$ is not invertible, otherwise Theorem \ref{partiallylocalisedthm} follows from Theorem \ref{localisedthm}.

Fix $\alpha\in(0,1)$ and let
\begin{equation}\label{defc}
c:=\min_{\ell}c_\ell,\quad \delta:=\min_{\ell}\delta_\ell,
\end{equation}
where $c_\ell, \delta_\ell$ are those given by Lemma \ref{new51}. We emphasise that these quantities depend only on $H_0$,  $\alpha$,  and the fixed datum $(\linearfamily^0,\mathbf{p})$.

It will suffice to prove that there exists a constant $C>0$ such that
\begin{equation*}
 \int_{H} e^{-\pi\langle Gx,x\rangle} \prod_{j=1}^m (f_j \circ \linearmapj)^{p_j}
 \leq C \prod_{j=1}^m \bigg(\int_{H_j} f_j \bigg)^{p_j}
\end{equation*}
whenever $\|\linearfamily-\linearfamily^0\|<\delta$. By Lieb's Theorem (Theorem 6.2 in \cite{L}), this is equivalent to proving that
\begin{equation} \label{e:Liebgaussian}
\prod_{j=1}^m (\det A_j)^{p_j} \leq C \det (M+G)
\end{equation}
holds uniformly for such $\linearfamily$ and all positive definite $A_j:H_j\rightarrow H_j$, $1\leq j\leq m$, where $M=\sum_{j=1}^m p_j\linearmapj^*A_j\linearmapj$.
%By an application of Lieb's Theorem (Theorem 6.2 in \cite{L}), we have
%$$
%\mbox{BL}_G(\linearfamily,\mathbf{p})\leq\sup\;\frac{\prod_{j=1}^m(\det_{H_j}A_j)^{p_j/2}}{\det(G+M)^{1/2}},
%$$
%where $M=\sum_{j=1}^m p_j\linearmapj^*A_j\linearmapj$, and the supremum is taken over all positive definite $A_j:H_j\rightarrow H_j$, $1\leq j\leq m$.
%It will thus suffice to prove that there exists a constant $C>0$ such that
%$$ \prod_{j=1}^m (\det A_j)^{p_j} \leq C \det (M+G)$$
%for all data $(\linearfamily,\mathbf{p})$ such that $\|\linearfamily-\linearfamily^0\|\leq\delta$ and all such positive definite $A_j$.
To this end we fix an auxiliary quantity
%$$\gamma =\min\Biggl\{\Biggl(\frac{1-\alpha}{n}\Biggr)^2,\Biggl(\frac{c}{2 n(\max_{j}\|\linearmapj^0\|+\delta)}\Biggr)^2\Biggr\}.$$
$$
\gamma =\min\Biggl\{\bigg(\frac{1-\alpha}{n}\bigg)^2,\bigg(\frac{c}{2 \max_{j} n_j(\|\linearmapj^0\|+\delta)^{n_j}}\bigg)^2\Biggr\},
$$
which, of course, only depends on the fixed datum $(\linearfamily^0,\mathbf{p})$, $\delta$ and our choice of $\alpha$.

Since $p_j>0$ and $\bigcap_{j=1}^m \ker{\linearmapj}=\{0\}$, we have that $M$, and thus $M+G$, is positive definite. Let $e_1,\hdots,e_n$ be an orthonormal basis of eigenvectors for $M+G$, ordered so that their corresponding eigenvalues satisfy $\mu_1 \geq \cdots \geq \mu_n > 0$. As in the proof of Theorem \ref{localisedthm}, we have that $\langle e_i, Me_i \rangle_{H} \leq \mu_i$ and $\langle A_j \linearmapj e_i, \linearmapj e_i \rangle_{H_j} \leq  \frac{\mu_i}{p_j}$ for each $1 \leq i \leq n$ and $1 \leq j \leq m$.

Suppose first that $\mu_n\geq \gamma$. This case is much the same as the fully localised case, except that the lower bound on the eigenvalues is $\gamma$ rather than $1$. Thus, by rescaling and using the argument in the localised case, it follows that
$$
\prod_{j=1}^m (\det A_j)^{p_j} \leq \gamma^{(\sum_{j=1}^m p_jn_j) - n}\bigg(c^{2\sum_{j=1}^mp_j}\prod_{j=1}^m p_j^{p_jn_j}\bigg)^{-1}\det (M+G).
$$

Otherwise $\mu_n<\gamma$. In this case, the argument combines the approaches from the localised  and the non-localised cases in \cite{BCCT1}. We define
$$
\ell := \min\{i\in\{1,\hdots,n\}: \mu_i<\gamma\}.
$$
Since $G$ is an orthogonal projection,
\begin{equation*}
|Ge_i|^2=\langle Ge_i,e_i\rangle\leq \langle (M+G)e_i,e_i\rangle\leq \mu_i\leq\gamma\leq \left(\frac{1-\alpha}{n}\right)^2
\end{equation*}
whenever $\ell\leq i\leq n$. Consequently, using the multilinearity of the wedge product,
\begin{equation}\label{firstuse}
\left|\bigwedge_{i=\ell}^n (e_i-Ge_i)-\bigwedge_{i=\ell}^n e_i\right|\leq (n-\ell+1)\max_{\ell\leq i\leq n}|Ge_i|\leq 1-\alpha,
\end{equation}
so that, by the triangle inequality and the orthonormality of the vectors $e_\ell,\hdots,e_n$,
$$
\left|\bigwedge_{i=\ell}^n (e_i-Ge_i)\right|\geq \alpha>0.
$$
Thus the vectors $$e_\ell-Ge_\ell,\hdots,e_n-Ge_n$$ are linearly independent in $H_0$, and in particular we have $\ell \geq n - \dim(H_0) + 1$.

Next, define $(v_1,\hdots,v_n)$ by $v_i:=e_i$ for $1\leq i \leq\ell - 1$ and by $v_i:=e_i-Ge_i \in H_0$ for $\ell\leq i\leq n$.  By definition,  $\|v_i\|\leq1$ for all $1\leq i \leq n$.  Arguing as in \eqref{firstuse},
$$
\bigg|  \bigwedge_{i=1}^n  v_i -  \bigwedge_{i=1}^n  e_i \bigg|  \leq (n-\ell + 1) \max_{1 \leq i \leq n} \|e_i-v_i\|\leq n\gamma^{1/2}\leq 1-\alpha,
$$
and therefore
$$
\bigg|  \bigwedge_{i=1}^n  v_i \bigg|  \geq 1- n\gamma^{1/2}\geq \alpha.
$$
Whence, $(v_1,\hdots,v_n)\in\mathcal{V}_{\alpha,\ell}$.

By Lemma \ref{new51} applied to $\mathbf{v}=(v_1,\hdots,v_n)$, there exists $I_j \subseteq \{1,\ldots,n\}$ for each $1 \leq j \leq m$  of cardinality $|I_j| = \dim(H_j)$
 such that  all three of \eqref{35b}, \eqref{35ker}, and \eqref{sizeb} hold for $\mathbf{v}$.

%By our choice of $(v_1,\hdots,v_n)$ similar conditions will hold in turn for $(e_1,\hdots,e_n)$. The only condition that needs to be verified is \eqref{sizeb}. Again by multilinearity and the choice of $\gamma$,
Now observe that
$$
\bigg|  \bigwedge_{i \in I_j}  \linearmapj v_i -  \bigwedge_{i \in I_j}  \linearmapj e_i \bigg|  \leq n_j \|L_j\|^{n_j} \max_{i\in I_j} \|e_i-v_i\| \leq n_j \|L_j\|^{n_j} \gamma^{1/2},
$$
and hence our choice of $\gamma$ guarantees that
$$
\bigg|  \bigwedge_{i \in I_j}  \linearmapj e_i\bigg| \geq c- \max_j n_j\|\linearmapj\|^{n_j}\gamma^{1/2}\geq c/2
$$
whenever $\|\linearmapj-\linearmapj^0\|\leq\delta$.
%since $\|\linearmapj-\linearmapj^0\|\leq\delta$.
Here we are assuming that we have chosen norms so that $\|\linearmapj\|$ is bounded above by $\|\linearfamily\|$. This is of course quite natural, although since all choices of norms are equivalent, there is no loss of generality in doing this.
As before, we set $a_i := \sum_{j=1}^m p_j |I_j \cap \{i\}|$ and obtain
$$
\prod_{j=1}^m (\det A_j)^{p_j} \leq \Biggl((c/2)^{2\sum_{j=1}^mp_j}\prod_{j=1}^m p_j^{p_jn_j}\Biggr)^{-1} \prod_{i=1}^n \mu_i^{a_i}
$$
and a telescoping argument yields
$$ \prod_{i=1}^{\ell-1} \mu_i^{a_i}\leq \gamma^{(\sum_{i=1}^{\ell-1}a_i)-(\ell-1)} \prod_{i=1}^{\ell-1} \mu_i.$$

For the terms with $i \geq \ell$,
%$$ \prod_{i=\ell}^n \mu_i^{\sum_{j=1}^m p_j |I_j \cap \{i\}|}.$$
similarly to the global case in \cite{BCCT1}, we write
$$
\prod_{i=\ell}^n \mu_i^{a_i}=\mu_\ell^{a_{\geq \ell}}\prod_{i=\ell}^{n-1} \left(\frac{\mu_{i+1}}{\mu_i}\right)^{a_{\geq i+1}}.
$$
where $a_{\geq \ell} := \sum_{j=1}^m p_j |I_j \cap \{\ell,\ldots,n\}|$.

%$$
%\prod_{i=\ell}^n \mu_i^{a_i}=\mu_\ell^{\sum_{j=1}^m p_j |I_j \cap \{\ell,\ldots,n\}|}\prod_{i=\ell}^n \left(\frac{\mu_{i+1}}{\mu_i}\right)^{\sum_{j=1}^m p_j |I_j \cap \{i+1,\ldots,n\}|}.
%$$

By \eqref{35ker}, for $\ell \leq i \leq n-1$ we have
$
a_{\geq i+1} \geq n-i
$
and as $\frac{\mu_{i+1}}{\mu_i}\leq 1$, this yields
$$
\prod_{i=\ell}^n \mu_i^{a_i}\leq \mu_\ell^{a_{\geq \ell}} \prod_{i=\ell}^{n-1} \left(\frac{\mu_{i+1}}{\mu_i}\right)^{n-i},
$$
which on reversing the telescoping gives,
$$
\prod_{i=\ell}^n \mu_i^{a_i}\leq \mu_\ell^{a_{\geq \ell}-(n-\ell+1)}\prod_{i=\ell}^n \mu_i.
$$
Recall that $\ell - 1\geq n-\dim(H_0)$, so that we may apply \eqref{35ker} to conclude that $a_{\geq \ell}\geq n-\ell+1$, and therefore
$
\prod_{i=\ell}^n \mu_i^{a_i}\leq \prod_{i=\ell}^n \mu_i.
$
Finally, we obtain
\begin{equation}\label{finaleq} \prod_{j=1}^m (\det A_j)^{p_j} \leq  \gamma^{(\sum_{i=1}^{\ell-1}a_i)-(\ell-1)}  \bigg((c/2)^{2\sum_{j=1}^mp_j}\prod_{j=1}^m p_j^{p_jn_j}\bigg)^{-1} \det (M+G)
\end{equation}
which concludes the proof.
\end{proof}
\section{The proofs of Theorems \ref{CBLapp}--\ref{NLBLapp}}\label{Sectionappl}
As we shall see in this section, the local boundedness of the Brascamp--Lieb constant established in Theorem \ref{mainthm} is a natural requirement for the induction-on-scales method to yield Theorems \ref{CBLapp}--\ref{NLBLapp}. Within harmonic analysis at least, the induction-on-scales arguments that we use go back to Bourgain \cite{Bo}, and have been used extensively since; see in particular \cite{BB}, \cite{BHT}, \cite{BCT}, \cite{B}, \cite{Guth2} for very similar arguments in the context of the Loomis--Whitney and multilinear Kakeya inequalities. In this Brascamp--Lieb setting, these inductive arguments are manifestations of a fundamental multi-scale inequality of Ball \cite{Ball}, and are closely related to heat-flow monotonicity and semigroup interpolation; see \cite{CLL}, \cite{BCCT1}, \cite{BCT}, \cite{BB} for further discussion of this perspective.
%As indicated in the introduction, we omit the proof of Theorem \ref{OBLapp} as it is a routine generalisation of the argument presented in \cite{BCT} (see also \cite{B}) in the context of the Loomis--Whitney datum.

We warn that the function $\mathcal{C}$ will have a different definition in each of the sections \ref{CBLsect}--\ref{NLBLsect} below.

\subsection{Generalised multilinear Kakeya inequalities and Theorem \ref{CBLapp}}\label{CBLsect}
Here we prove Theorem \ref{CBLapp} using the induction-on-scales argument in Guth \cite{Guth2}. The role of Theorem \ref{mainthm} in this argument is to effectively change the order of the quantifiers in the hypothesis of Theorem \ref{CBLapp}. By suitably partitioning the families $\mathbb{T}_j$, $1\leq j\leq m$, and applying Theorem 1.1, it suffices to prove the following
weaker variant of Theorem \ref{CBLapp}. The deduction of Theorem \ref{CBLapp} in this way incurs a cost in the size of the constant $C$, but in a way which only depends on $\varepsilon$.
\begin{theorem}\label{CBLappprime}
Suppose $(\linearfamily,\mathbf{p})$ is a Brascamp--Lieb datum for which $\mbox{\emph{BL}}(\linearfamily,\mathbf{p})<\infty$, and $\varepsilon>0$. Then there exists $\nu=\nu(\varepsilon)>0$ and $C=C(\varepsilon)<\infty$ (both independent of $\delta$) such that
\begin{equation*}%\label{combBLweak}
\int_{[-1,1]^n}\prod_{j=1}^m\Bigl(\sum_{T_j\in\mathbb{T}_j}\chi_{T_j}\Bigr)^{p_j} \leq C\delta^{n-\varepsilon}
\prod_{j=1}^m\left(\#\mathbb{T}_j\right)^{p_j}
\end{equation*}
holds for all finite collections $\mathbb{T}_j$ of $\delta$-neighbourhoods of $n_j'$-dimensional affine subspaces of $\mathbb{R}^n$ which, modulo translations, are within a distance $\nu$ of the fixed subspace $V_j:=\ker L_j$.
\end{theorem}

For each $0<\delta,\nu\leq 1$, let $\mathcal{C}(\delta, \nu)$ denote the smallest constant $C$ in the inequality
\begin{equation}\label{CBLconst}
\int_{[-1,1]^n}\prod_{j=1}^m\Bigl(\sum_{T_j\in\mathbb{T}_j}\chi_{T_j}\Bigr)^{p_j}\leq C\delta^{n}\prod_{j=1}^m(\# \mathbb{T}_j)^{p_j}
\end{equation}
over all such families $\mathbb{T}_j$, $1\leq j\leq m$, as in the statement of Theorem \ref{CBLappprime}. We are required to show that given any $\varepsilon>0$, there exists $\nu=\nu(\varepsilon)>0$ (independent of $\delta$) such that $\mathcal{C}(\delta,\nu)\lesssim_\varepsilon\delta^{-\varepsilon}$.
\begin{proposition}\label{recursiveGuth}
There is a constant $\kappa<\infty$, independent of $\delta$ and $\nu$, such that $$\mathcal{C}(\delta,\nu)\leq\kappa\mathcal{C}(\delta/\nu, \nu).$$
\end{proposition}
Iterating Proposition \ref{recursiveGuth} we obtain $\mathcal{C}(\delta,\nu)\leq\kappa^\ell\mathcal{C}(\delta/\nu^\ell,\nu)$ for each $\ell\in\mathbb{N}$. We choose $\ell$ such that $\delta/\nu^\ell\sim 1$, and $\nu$ such that $\varepsilon \log(1/\nu) = \log \kappa$, so that $\kappa^\ell \lesssim_\varepsilon \delta^{-\varepsilon}$ and hence $\mathcal{C}(\delta,\nu)\lesssim_\varepsilon\delta^{-\varepsilon}$, as required.

\begin{proof}[Proof of Proposition \ref{recursiveGuth}]
We begin by decomposing $[-1,1]^n$ into a grid of axis-parallel cubes $Q$ of sidelength $\delta/\nu$ and write
$$\int_{[-1,1]^n}\prod_{j=1}^m\Bigl(\sum_{T_j\in\mathbb{T}_j}\chi_{T_j}\Bigr)^{p_j}
=\sum_Q\int_Q\prod_{j=1}^m\Bigl(\sum_{T_j\in\mathbb{T}_j(Q)}\chi_{T_j\cap Q}\Bigr)^{p_j},
$$
where $\mathbb{T}_j(Q):=\{T_j\in\mathbb{T}_j:T_j\cap Q\not=\emptyset\}$. Observe that since $T_j$ is a $\delta$-neighbourhood of an affine $n_j'$-dimensional subspace of $\mathbb{R}^n$ which, modulo translations, is within a distance $\nu$ of $V_j:=\ker L_j$, there exists an $O(\delta)$-neighbourhood $T_j'$ with $T_j\cap Q\subseteq T_j'\cap Q$ and $T_j'$ parallel to $V_j$. Since $\mbox{BL}(\linearfamily,\mathbf{p})<\infty$,
\begin{eqnarray*}
\begin{aligned}
\int_Q\prod_{j=1}^m\Bigl(\sum_{T_j\in\mathbb{T}_j(Q)}\chi_{T_j\cap Q}\Bigr)^{p_j}&\leq\int_Q\prod_{j=1}^m\Bigl(\sum_{T_j\in\mathbb{T}_j(Q)}\chi_{T_j'}\Bigr)^{p_j}\\
&\lesssim\delta^n\prod_{j=1}^m(\#\mathbb{T}_j(Q))^{p_j}\\
&\leq\delta^n\prod_{j=1}^m\Bigl(\sum_{T_j\in\mathbb{T}_j}\chi_{\widetilde{T}_j}(x_Q)\Bigr)^{p_j}
\end{aligned}
\end{eqnarray*}
uniformly in $x_Q\in Q$. Here $\widetilde{T}_j=T_j+B(0,c\delta/\nu)$, with factor $c$ chosen large enough so that $T_j\cap Q\not=\emptyset\Rightarrow Q\subseteq\widetilde{T}_j$.
Averaging the above over $x_Q\in Q$ gives
$$
\int_Q\prod_{j=1}^m\Bigl(\sum_{T_j\in\mathbb{T}_j(Q)}\chi_{T_j\cap Q}\Bigr)^{p_j}\lesssim\nu^n\int_Q\prod_{j=1}^m\Bigl(\sum_{T_j\in\mathbb{T}_j}\chi_{\widetilde{T}_j}\Bigr)^{p_j},$$
which on summing in $Q$ results in
\begin{eqnarray*}
\begin{aligned}
\int_{[-1,1]^n}\prod_{j=1}^m\Bigl(\sum_{T_j\in\mathbb{T}_j(Q)}\chi_{T_j\cap Q}\Bigr)^{p_j}&\lesssim\nu^n\int_{[0,1]^n}\prod_{j=1}^m\Bigl(\sum_{T_j\in\mathbb{T}_j}\chi_{\widetilde{T}_j}\Bigr)^{p_j}\\
&\lesssim\delta^n\mathcal{C}(\delta/\nu,\nu)\prod_{j=1}^m(\#\mathbb{T}_j)^{p_j},
\end{aligned}
\end{eqnarray*}
from which the proposition follows.
\end{proof}

\subsection{Generalised multilinear restriction inequalities and Theorem \ref{OBLapp}}\label{OBLsect}
The deduction of Theorem \ref{OBLapp} from Theorem \ref{CBLapp} is a routine generalisation of the argument in \cite{BCT} (see also \cite{B}) in the setting of the Loomis--Whitney datum. We provide a sketch of the argument here for the sake of completeness.
\begin{proof}[Proof of Theorem \ref{OBLapp} from Theorem \ref{CBLapp}]
We begin with an observation.
For each $\varepsilon > 0$ and $R \geq 1$, applying Theorem \ref{CBLapp} with $\delta = R^{-1/2}$, and using rescaling and limiting arguments we obtain
\begin{equation} \label{e:Kakeyarescaled}
\int_{B(0,R)} \prod_{j=1}^m \bigg(\sum_{T_{j,R} \in \mathbb{T}_{j,R}} \frac{\chi_{T_j}}{|T_j|} * g_{T_j} \bigg)^{p_j} \lesssim_\varepsilon R^{\frac{\varepsilon}{2}-\sum_{j=1}^m p_jn_j'} \prod_{j=1}^m \bigg(\sum_{T_j \in \mathbb{T}_{j,R}} \|g_{T_j}\|_1\bigg)^{p_j}
\end{equation}
for all nonnegative $g_{T_j} \in L^1(\mathbb{R}^n)$, $T_j \in \mathbb{T}_{j,R}$, $1 \leq j \leq m$. Here $\mathbb{T}_{j,R}$ is any finite collection of rectangles in $\mathbb{R}^n$ with $n_j$ sides of length $O(R^{1/2})$ and $n_j'$ sides of length $O(R)$, with the property that each $T_j\in\mathbb{T}_{j,R}$ is contained in an $O(R^{1/2})$-neighbourhood of an $n_j'$-dimensional subspaces of $\mathbb{R}^n$ which is (modulo translations) within a distance $\nu > 0$ (given by Theorem \ref{CBLapp}) of $\ker L_j$.

In order to prove Theorem \ref{OBLapp} it will suffice to show that
\begin{equation}\label{wantrest}
\int_{B(0,R)} \prod_{j=1}^m |G_j|^{2p_j} \lesssim_\varepsilon R^{\varepsilon - \sum_{j=1}^mp_jn_j'} \prod_{j=1}^m \|G_j\|_{2}^{2p_j}
\end{equation}
for all $G_j \in L^2(\mathbb{R}^n)$ such that $\mathrm{supp} \,\widehat{G_j} \subseteq S_j + O(R^{-1})$, $1 \leq j \leq m$, and all $R \geq 1$. To see that \eqref{wantrest} implies \eqref{genmultrestloss} we first observe that $E_jg_j=\widehat{h_jd\sigma_j}$, where the $S_j$-carried measure $\sigma_j$ is defined by
$$
\int_{\mathbb{R}^n}\psi d\sigma_j:=\int_{U_j}\psi(\Sigma_j(x))dx
$$
and $h_j$ by $g_j=h_j\circ\Sigma_j$. Let $\phi$ be a smooth bump function supported in $B(0,1)$ with Fourier transform bounded below on $B(0,1)$, and let $\phi_R(x)=R^n\phi(Rx)$. Setting $G_j=(E_jg_j)\widehat{\phi}_R$ reveals that
$\mathrm{supp} \,\widehat{G_j} \subseteq S_j + O(R^{-1})$ and $\|G_j\|_2\sim R^{n_j'/2}\|h_j\|_2\sim R^{n_j'/2}\|g_j\|_2$ uniformly in $R$ for each $1\leq j\leq m$. Applying \eqref{wantrest} to these functions $G_j$ establishes \eqref{genmultrestloss}; see \cite{TVV} for further details of this reduction in a bilinear setting.

Next we let $\mathcal{C}(R)$ denote the smallest constant $C$ in the inequality
\begin{equation}\label{wantrest1}
\int_{B(0,R)} \prod_{j=1}^m |G_j|^{2p_j} \leq C R^{- \sum_{j=1}^mp_jn_j'} \prod_{j=1}^m \|G_j\|_{2}^{2p_j}
\end{equation}
over all $G_j \in L^2(\mathbb{R}^n)$ such that $\mathrm{supp} \,\widehat{G_j} \subseteq S_j + O(R^{-1})$, $1 \leq j \leq m$. In these terms \eqref{wantrest} becomes $\mathcal{C}(R)\lesssim_\varepsilon R^\varepsilon$.

Upon iterating and using the elementary fact that $\mathcal{C}(100)<\infty$, it will
be enough to prove that for each $\varepsilon>0$, there exists a constant $c_\varepsilon$, independent of $R$, such that
\begin{equation}\label{restrec}\mathcal{C}(R)\leq c_\varepsilon R^\varepsilon\mathcal{C}(R^{1/2})
\end{equation}
for all $R\geq 1$.

Let $x\in B(0,R)$ and $\phi_{R^{1/2}}^x:\mathbb{R}^n\rightarrow\mathbb{R}$ be given by $\phi_{R^{1/2}}^x(y)=e^{-2\pi ix \cdot y}\phi_{R^{1/2}}(y)$, where $\phi_{R^{1/2}}$ is defined above; observe that the Fourier transform of $\phi^x_{R^{1/2}}$ is bounded below on $B(x,R^{1/2})$ uniformly in $x$ and $R$.
Applying \eqref{wantrest1} on $B(x,R^{1/2})$, using the modulation-invariance of the inequality, we obtain
$$
\int_{B(x,R^{1/2})}\prod_{j=1}^m |G_j|^{2p_j} \lesssim\mathcal{C}(R^{1/2})R^{-\frac{1}{2}\sum_{j=1}^m p_jn_j'}
% R^{ - \frac{1}{2}\sum_{j=1}^m p_jn_j'}
\prod_{j=1}^m \|\widehat{G}_j*\phi_{R^{1/2}}^x\|_2^{2p_j}
$$
uniformly in $x$ and $R$. Averaging this over all $|x|\leq R$  yields
$$
\int_{B(0,R)}\prod_{j=1}^m |G_j|^{2p_j} \lesssim\mathcal{C}(R^{1/2})R^{-\frac{n}{2}-\frac{1}{2}\sum_{j=1}^m p_jn_j'}%R^{ - \frac{1}{2}\sum_{j=1}^m p_jn_j-\frac{n}{2}}
 \int_{B(0,R)}\prod_{j=1}^m \|\widehat{G}_j*\phi_{R^{1/2}}^x\|_2^{2p_j}dx.
$$
Defining $\widehat{G^{\rho_j}_j} = \widehat{G_j}\chi_{\rho_j}$ for caps $\rho_j$ with diameter $R^{-1/2}$ which together provide a cover of $S_j + O(R^{-1})$ with bounded overlap, we may write
\begin{equation*}
\int_{B(0,R)} \prod_{j=1}^m |G_j|^{2p_j} \lesssim\mathcal{C}(R^{1/2})R^{-\frac{n}{2}-\frac{1}{2}\sum_{j=1}^m p_jn_j'}
% R^{ - \frac{1}{2}\sum_{j=1}^m p_jn_j' - \frac{n}{2}}
 \int_{B(0,R)} \prod_{j=1}^m \bigg( \sum_{\rho_j} \|\widehat{G^{\rho_j}_j}*\phi_{R^{1/2}}^x \|_2^2 \bigg)^{p_j} dx.
\end{equation*}
Using the rapid decay of the function $\phi_{R^{1/2}}^x$ it now suffices to show that
\begin{equation}\label{laststep}
\int_{B(0,R)} \prod_{j=1}^m \bigg( \sum_{\rho_j} \|G^{\rho_j}_j \|_{L^2(B(x,R^{1/2}))}^2 \bigg)^{p_j} dx\lesssim R^{\varepsilon+\frac{n}{2}-\frac{1}{2}\sum_{j=1}^m p_jn_j'}\prod_{j=1}^m \|G_j\|_{2}^{2p_j}.
\end{equation}
Now let $\widetilde{G}^{\rho_j}_j$ be given by $G_j^{\rho_j} = \widetilde{G}^{\rho_j}_j * \widehat{\psi}_{\rho_j}$, where $\psi_{\rho_j}$ is a Schwartz function which satisfies $\psi_{\rho_j} \sim 1$ on $\rho_j$ and
\begin{equation*}
|\widehat{\psi}_{\rho_j}(x + y)| \lesssim \frac{\chi_{\rho^*_j}(x)}{|\rho^*_j|}
\end{equation*}
uniformly in $x \in \mathbb{R}^n, y \in B(0,R^{1/2})$.
Here $\rho^*_j$ is a rectangle in $\mathbb{R}^n$ with $n_j$ sides of length $O(R^{1/2})$ and $n_j'$ sides of length $O(R)$, lying in an $O(R^{1/2})$-neighbourhood of an $n_j'$-dimensional affine subspace of $\mathbb{R}^n$, which, if the neighbourhoods $U_j$ are chosen sufficiently small, is within distance $\nu$ of $\ker L_j$ (modulo translations). Applying the Cauchy--Schwarz inequality and integrating in $y\in B(0,R^{1/2})$, we have
\begin{equation*}
\|G^{\rho_j}_j \|_{L^2(B(x,R^{1/2}))}^2 \lesssim R^{n/2} \frac{\chi_{\rho^*_j}}{|\rho^*_j|} * |\widetilde{G}^{\rho_j}_j|^2 (x)
\end{equation*}
uniformly in $\rho_j$, $x$ and $R$. An application of
\eqref{e:Kakeyarescaled}, the scaling condition \eqref{scaling}, followed by the bounded overlap property of the caps $\rho_j$, completes the proof of \eqref{laststep}, and hence \eqref{restrec}.
\end{proof}
It is interesting to note that in the particular case of the Loomis--Whitney datum, one can recover Theorem \ref{CBLapp} from Theorem \ref{OBLapp} by a simple Rademacher function argument (see \cite{BCT}). However, when the codimensions $n_j'\not=1$ this argument fails due to certain orientation restrictions on the dual objects $\rho_j^*$ arising in the above wave-packet analysis.

\subsection{Nonlinear Brascamp--Lieb inequalities and Theorem \ref{NLBLapp}}\label{NLBLsect}
The induction-on-scales argument we use here is very closely related to the one used in Section \ref{CBLsect}.
We begin by introducing, for each $0<\delta\leq 1$ and $n\in\mathbb{N}$, the class of functions
\begin{equation}\label{harnack}
L^1(\mathbb{R}^n;\delta):=\{f\in L^1(\mathbb{R}^n): f\geq 0\mbox{ and } \tfrac{1}{2}f(y)\leq f(x)\leq 2 f(y) \mbox{ whenever } |x-y|\leq\delta\}.
\end{equation}
As remarked in \cite{BB}, $u(c\delta,\cdot):=P_{c\delta}*\mu\in L^1(\mathbb{R}^n;\delta)$ for every finite Borel measure $\mu$ on $\mathbb{R}^n$, where $P_t$ denotes the Poisson kernel and $c$ a suitable dimensional constant. The defining property of a function $f\in L^1(\mathbb{R}^n;\delta)$ states that $f$ is ``essentially constant at scale $\delta$", and in the context of the harmonic function, $u$, is a manifestation of the Harnack principle.
\begin{proposition}\label{NLBLenough}
Under the hypotheses of Theorem \ref{NLBLapp}, there exists a neighbourhood $U$ of the origin in $\mathbb{R}^n$ and a constant $\kappa<\infty$ such that
\begin{equation}
\int_{U}\prod_{j=1}^m (f_j\circ \nonlinearmapj)^{p_j}\lesssim \left(\log\left(\frac{1}{\delta}\right)\right)^\kappa\prod_{j=1}^m
\left(\int_{\mathbb{R}^{n_j}}f_j\right)^{p_j}
\end{equation}
for all functions $f_j\in L^1(\mathbb{R}^{n_j};\delta)$, $1\leq j\leq m$, and all $0<\delta\leq 1$.
\end{proposition}
Before proving Proposition \ref{NLBLenough} we indicate how it implies Theorem \ref{NLBLapp}.
We begin with a simple observation. For each $1\leq j\leq m$ let $\psi_j$ be a Schwartz function on $\mathbb{R}^{n_j}$, and for each $\delta_1,\hdots,\delta_m\geq\delta>0$, let $\psi_{j,\delta_j}(x):=\delta_j^{-n_j}\psi_j(\delta_j^{-1}x)$. Bounding $|\psi_j|$ by a suitably normalised Poisson kernel, as we may, it follows that for each nonnegative $g_j\in L^1(\mathbb{R}^{n_j})$ there is a $\widetilde{g}_j\in L^1(\mathbb{R}^{n_j};\delta)$ such that
$|\psi_{j,\delta_j}|*g_j\lesssim \widetilde{g}_j$ and $\int\widetilde{g_j}\lesssim\int g_j$, with implicit constants uniform in $\delta_1,\hdots,\delta_m$ and $\delta$. Thus by Proposition \ref{NLBLenough},
\begin{equation}\label{touse}
\int_U\prod_{j=1}^m\left((|\psi_{j,\delta_j}|*g_j)\circ B_j\right)^{p_j}\lesssim \left(\log\left(\frac{1}{\delta}\right)\right)^\kappa\prod_{j=1}^m
\left(\int_{\mathbb{R}^{n_j}}g_j\right)^{p_j}
\end{equation}
for all nonnegative $g_j\in L^1(\mathbb{R}^{n_j})$, $1\leq j\leq m$, and $\delta_1,\hdots,\delta_m\geq\delta>0$.

Let $\varepsilon>0$. For each $1\leq j\leq m$ let $\{P_{j,k}\}_{k=0}^\infty$ be the standard annular Littlewood--Paley projection operators on $\mathbb{R}^{n_j}$ with associated convolution kernels $\{\phi_{j,k}\}_{k=0}^\infty$. We choose these kernels such that for $k>0$, $\widehat{\phi}_{j,k}(\xi)=\widehat{\phi}_j(2^{-k}\xi)$ for some fixed Schwartz function $\phi_j$ on $\mathbb{R}^{n_j}$ with Fourier support in the annulus $\{\xi\in\mathbb{R}^{n_j}:1/4\leq |\xi|\leq 2\}$, and such that $\phi_{j,0}$ is a Schwartz function with Fourier support in the unit ball of $\mathbb{R}^{n_j}$. Furthermore the functions $\{\widehat{\phi}_{j,k}\}_{k=0}^\infty$ are taken to form a partition of unity on $\mathbb{R}^{n_j}\backslash\{0\}$, so that $\sum_{k \geq 0}P_{j,k}$ is the identity for each $j$.
For each $1\leq j\leq m$ let $\widetilde{\phi}_{j}$ be a Schwartz function whose Fourier transform is equal to $1$ on the Fourier support of $\phi_j$, and define $\widetilde{\phi}_{j,k}$ in a similar way to $\phi_{j,k}$. Observe that $\widetilde{\phi}_{j,k}*\phi_{j,k}=\phi_{j,k}$ for all $j,k$.

We may thus write
$$
\int_U\prod_{j=1}^m f_j\circ B_j=\sum_{k_1,\hdots, k_m \geq 0}\int_U\prod_{j=1}^m(P_{j,k_j}f_j)\circ B_j.
$$
By symmetry we need only consider the above sum for $k_1\geq k_2\geq\cdots\geq k_m\geq 0$. Writing $P_{j,k_j}f_j=\widetilde{\phi}_{j,k_j}*(P_{j,k_j}f_j)$ and applying H\"older's inequality we have
\begin{eqnarray*}
\begin{aligned}
\int_U\prod_{j=1}^m|P_{j,k_j}f_j|\circ B_j&\leq\int_U\left(|\widetilde{\phi}_{j,k_j}|*|P_{j,k_j}f_j|\right)\circ B_j\\
&\lesssim\int_U\prod_{j=1}^m\left(|\widetilde{\phi}_{j,k_j}|*|P_{j,k_j}f_j|^{q_j}\right)^{p_j}\circ B_j,
\end{aligned}
\end{eqnarray*}
which by \eqref{touse} is, up to a bounded factor, bounded above by
$$
2^{\varepsilon k_1/2}\prod_{j=1}^m\|P_{j,k_j}f_j\|_{q_j}.
$$
Thus
\begin{eqnarray*}
\begin{aligned}
\sum_{k_1\geq\cdots\geq k_m \geq 0}\int_U\prod_{j=1}^m|P_{j,k_j}|\circ B_j&\lesssim \sum_{k_1\geq\cdots\geq k_m \geq 0}2^{\varepsilon k_1/2}\prod_{j=1}^m\|P_{j,k_j}f_j\|_{q_j}\\
& \lesssim \sum_{k_1\geq\cdots\geq k_m \geq 0} 2^{-\varepsilon k_1/2} \|f_1\|_{L^{q_1}_\varepsilon} \prod_{j=2}^m \|f_j\|_{q_j} \\
&\lesssim \sum_{k_1\geq\cdots\geq k_m \geq 0}2^{-\varepsilon k_1/(2m)}\cdots2^{-\varepsilon k_m/(2m)}\|f_1\|_{L^{q_1}_\varepsilon(\mathbb{R}^{d_1})}\prod_{j=2}^m\|f_j\|_{q_j}\\
&\lesssim\prod_{j=1}^m\|f_j\|_{L^{q_j}_\varepsilon(\mathbb{R}^{n_j})},
\end{aligned}
\end{eqnarray*}
as required.
% We now turn to the proof of Proposition \ref{NLBLenough}, which uses an induction-on-scales method going back at least as far as Bourgain (see also Ball); see also \cite{BB}, \cite{BHT}, \cite{BCT}, \cite{B} for very similar arguments in the context of the Loomis--Whitney and multilinear Kakeya inequalities.

\begin{proof}[Proof of Proposition \ref{NLBLenough}]
Let $\eta$ be a positive real number to be determined. For $\delta>0$ let $\mathcal{C}(\delta)$ denote the best constant $C$ in the inequality
\begin{equation}
\int_{[-\eta,\eta]^n}\prod_{j=1}^m (f_j\circ \nonlinearmapj)^{p_j}\leq C\prod_{j=1}^m
\left(\int_{\mathbb{R}^{n_j}}f_j\right)^{p_j}
\end{equation}
over all functions $f_j\in L^1(\mathbb{R}^{n_j};\delta)$, $1\leq j\leq m$.
Of course Proposition \ref{NLBLenough} states that for some choice of $\eta$, depending only on the nonlinear maps $B_1,\hdots, B_m$ and exponents $p_1,\hdots, p_m$, there is a $\kappa<\infty$ for which $\mathcal{C}(\delta)\lesssim\left(\log(1/\delta)\right)^\kappa$. This will follow upon iterating $O(\log\log(1/\delta))$ times the recursive inequality
\begin{equation}\label{NLBLrecursive}
\mathcal{C}(\delta)\lesssim\mathcal{C}(\sqrt{\delta}).
\end{equation}
There will be more than one constraint placed on $\eta$, although the most significant will be a consequence of the local boundedness of the classical Brascamp--Lieb constant, established in Theorem \ref{mainthm}. Since the $\nonlinearmapj$ are smooth in a neighbourhood of the origin, and $\mathrm{d}\nonlinearmapj(0)=\linearmapj $, we have that $\|\mathrm{d}\nonlinearmapj(x)-\linearmapj \|\lesssim |x|$ in this neighbourhood. Thus, by Theorem \ref{mainthm}, there exists $\eta_0>0$ such that $\mbox{BL}((\mathrm{d}\nonlinearmapj(x))_{j=1}^m, \mathbf{p})<\infty$ uniformly in $|x|\leq\eta_0$.

In order to prove \eqref{NLBLrecursive}, we first partition $[-\eta,\eta]^n$ into a disjoint union of axis-parallel cubes $Q$ of sidelength
$\sqrt{\delta}$, and write
\begin{equation*}
\int_{[-\eta,\eta]^n}\prod_{j=1}^m (f_j\circ \nonlinearmapj)^{p_j}=\sum_Q\int_Q\prod_{j=1}^m ((f_j\chi_{B_j(Q)})\circ \nonlinearmapj)^{p_j}.
\end{equation*}
Taylor expanding $\nonlinearmapj$ about $x_Q\in Q$ we obtain
$$
\nonlinearmapj(x)=\nonlinearmapj(x_Q)+\mathrm{d}\nonlinearmapj(x_Q)(x-x_Q)+O(|x-x_Q|^2),
$$
and so if $x\in Q$,
$$
\nonlinearmapj(x)-(\nonlinearmapj(x_Q)+\mathrm{d}\nonlinearmapj(x_Q)(x-x_Q))=O(\delta).
$$
Since $f_j\in L^1(\mathbb{R}^{n_j};\delta)$ we have that
$$
f_j(\nonlinearmapj(x))\lesssim f_j(\nonlinearmapj(x_Q)+\mathrm{d}\nonlinearmapj(x_Q)(x-x_Q)),
$$
uniformly in $x\in Q$ and $Q\subseteq [-\eta,\eta]^n$. By translation-invariance we have
\begin{eqnarray*}
\begin{aligned}
\int_Q\prod_{j=1}^m ((f_j\chi_{B_j(Q)})\circ \nonlinearmapj)^{p_j} %& \lesssim \int_{Q}\prod_{j=1}^m f_j(\nonlinearmapj(x_Q)+\mathrm{d}\nonlinearmapj(x_Q)(x-x_Q))^{p_j}dx\\
\lesssim \mbox{BL}((\mathrm{d}\nonlinearmapj(x_Q))_{j=1}^m,\mathbf{p})\prod_{j=1}^m\left(\int_{\nonlinearmapj(\widetilde{Q})}f_j\right)^{p_j}
\end{aligned}
\end{eqnarray*}
for all $Q$, where $\widetilde{Q}$ is a suitable centred dilate of $Q$. Choosing $\eta\leq\eta_0$ we obtain
\begin{eqnarray}
\begin{aligned}
\int_{[-\eta,\eta]^n}\prod_{j=1}^m (f_j\circ \nonlinearmapj)^{p_j}\lesssim\sum_Q\prod_{j=1}^m\left(\int_{\nonlinearmapj(\widetilde{Q})}f_j\right)^{p_j}.
\end{aligned}
\end{eqnarray}
Since $\mathrm{d}\nonlinearmapj(0)=\linearmapj :\mathbb{R}^n\rightarrow\mathbb{R}^{n_j}$ is a surjection, and $\nonlinearmapj$ is smooth in a neighbourhood of the origin, we have that (making $\eta>0$ smaller if necessary), $|\nonlinearmapj(\widetilde{Q})|\sim\delta^{n_j/2}$ and
$$
\frac{1}{|\nonlinearmapj(\widetilde{Q})|}\int_{\nonlinearmapj(\widetilde{Q})}f_j\lesssim P_{c\sqrt{\delta}}*f_j(\nonlinearmapj(x_Q))
$$
uniformly in $x_Q\in Q$ and $Q\subseteq [-\eta,\eta]^n$. Here, as before, $P_t$ denotes the Poisson kernel on $\mathbb{R}^{n_j}$ (the dimension dictated by context), and $c$ a dimensional constant to be chosen momentarily. Thus
$$
\int_{[-\eta,\eta]^n}\prod_{j=1}^m (f_j(\nonlinearmapj(x)))^{p_j}dx\lesssim\delta^{\frac{1}{2}(p_1n_1+\cdots+p_mn_m)}\sum_Q\prod_{j=1}^m(P_{c\sqrt{\delta}}*f_j(\nonlinearmapj(x_Q)))^{p_j}.
$$
Averaging in the choices of $x_Q\in Q$, and using the scaling condition $\sum_{j=1}^m p_jn_j=n$, yields
\begin{eqnarray*}
\begin{aligned}
\int_{[-\eta,\eta]^n}\prod_{j=1}^m (f_j(\nonlinearmapj(x)))^{p_j}dx&\lesssim\sum_Q\int_Q\prod_{j=1}^m(P_{c\sqrt{\delta}}*f_j(\nonlinearmapj(x)))^{p_j}dx\\
&\lesssim\int_{[-\eta,\eta]^n}\prod_{j=1}^m(P_{c\sqrt{\delta}}*f_j(\nonlinearmapj(x)))^{p_j}dx.
\end{aligned}
\end{eqnarray*}
Choosing $c>0$ appropriately ensures that $P_{c\sqrt{\delta}}*f_j\in L^1(\mathbb{R}^{n_j};\sqrt{\delta})$ for each $j$. The claimed inequality \eqref{NLBLrecursive} follows.
\end{proof}

%%%%%%%%%%%%%%%%%%%%%%%%%%%%%%%%%%%%%%%%%%%%%%%%%%%%%%%%%%%%%%%%%%%%%%%%%%%%%%%%%%%%%%%%%%%%
%\bibliographystyle{plain}

\end{document}